\documentclass[11pt,reqno]{amsart}

\usepackage{amsmath,amssymb,amsthm,mathtools}
\usepackage{geometry}
\usepackage{enumitem}
\usepackage{booktabs}
\usepackage{xcolor}
\usepackage[colorlinks=true,linkcolor=blue,citecolor=blue,urlcolor=blue]{hyperref}
\usepackage{aliascnt}
\usepackage[capitalise,noabbrev]{cleveref}
\usepackage{fancyhdr}

\hypersetup{
  pdftitle={Lower bounds for some value sets over finite fields: incidence geometry and Bourgain's group expansion theorem},
  pdfauthor={Xiyu Hu},
  pdfsubject={Value sets generated by structured sequences over prime fields},
  pdfcreator={pdflatex}
}
\setlist{itemsep=0pt,topsep=0.3em,parsep=0pt,partopsep=0pt}
\allowdisplaybreaks
\numberwithin{equation}{section}
\newtheorem{theorem}{Theorem}[section]

\newaliascnt{proposition}{theorem}
\newtheorem{proposition}[proposition]{Proposition}
\aliascntresetthe{proposition}

\newaliascnt{lemma}{theorem}
\newtheorem{lemma}[lemma]{Lemma}
\aliascntresetthe{lemma}

\newaliascnt{corollary}{theorem}
\newtheorem{corollary}[corollary]{Corollary}
\aliascntresetthe{corollary}

\theoremstyle{definition}
\newaliascnt{definition}{theorem}

\aliascntresetthe{definition}

\theoremstyle{remark}
\newaliascnt{remark}{theorem}
\newtheorem{remark}[remark]{Remark}
\aliascntresetthe{remark}

\crefname{theorem}{Theorem}{Theorems}
\Crefname{theorem}{Theorem}{Theorems}
\crefname{proposition}{Proposition}{Propositions}
\Crefname{proposition}{Proposition}{Propositions}
\crefname{lemma}{Lemma}{Lemmas}
\Crefname{lemma}{Lemma}{Lemmas}
\crefname{corollary}{Corollary}{Corollaries}
\Crefname{corollary}{Corollary}{Corollaries}
\crefname{definition}{Definition}{Definitions}
\Crefname{definition}{Definition}{Definitions}
\crefname{remark}{Remark}{Remarks}
\Crefname{remark}{Remark}{Remarks}

\newcommand{\F}{\mathbb F}
\newcommand{\Pone}{\mathbb P^1}
\newcommand{\Aff}{\operatorname{Aff}}
\newcommand{\SL}{\operatorname{SL}}
\newcommand{\PSL}{\operatorname{PSL}}
\newcommand{\GL}{\operatorname{GL}}
\newcommand{\id}{\operatorname{id}}

\newcommand{\ord}{\operatorname{ord}}

\newcommand{\cE}{\mathcal E}

\newcommand{\cL}{\mathcal L}
\newcommand{\cN}{\mathcal N}
\newcommand{\cP}{\mathcal P}
\newcommand{\cS}{\mathcal S}

\newcommand{\cV}{\mathcal V}

\title{Lower bounds for some value sets over finite fields:\\ incidence geometry and Bourgain's group expansion theorem}
\author{Xiyu Hu}
\address{School of Mathematical Sciences, University of Chinese Academy of Sciences}
\email{hxyqpr@gmail.com}
\date{}
\subjclass[2020]{11B50, 11B65, 11T23, 20D60, 52C10}
\keywords{value sets, factorial residues, $q$-factorials, derangements, binomial coefficients, Catalan numbers, M\"obius transformations, finite fields, incidence geometry, expansion in $\SL_2$}

\begin{document}

\begin{abstract}
We develop two transition principles for lower-bounding value sets generated by structured sequences over prime fields.  A reciprocal-affine family with $M$ internal transitions and bounded quotient multiplicity has image size $\gg \min\{M,p\}^{8/15}$.  This recovers the factorial-residue bound and yields the same exponent for arithmetic Pochhammer products, Gaussian $q$-factorials, derangement numbers, and the numbers of ordered subsets.  A second theorem treats nonzero sequences whose consecutive ratios evolve under a nondegenerate M\"obius transformation: their value sets have size $\gg \min\{M,p\}^{1/2+\eta}$ for an absolute constant $\eta>0$.  As consequences, fixed rows of Pascal's triangle and the initial half-blocks of the Catalan and central binomial sequences exceed the square-root scale.  The proofs combine transition quotients with, respectively, Cartesian-product point-line incidence geometry and Bourgain's expansion-based incidence theorem in $\SL_2(\F_p)$.
\end{abstract}

\maketitle
\tableofcontents

\section{Introduction}

\subsection{Structured sequences and the square-root barrier}

Let $p$ be a prime.  If $(u_n)$ is an integer sequence and $I$ is a finite set of indices, write
\begin{equation}\label{eq:value-set-definition}
  \cV_p(u;I)=\{u_n\bmod p:n\in I\}\subseteq\F_p.
\end{equation}
The value sets considered in this paper are not polynomial value sets of bounded degree.  Their entries are generated by products, factorial ratios, or inhomogeneous recurrences whose arithmetic complexity grows with the index.  The useful structure is therefore not a fixed polynomial formula for $u_n$, but a low-complexity relation among nearby terms.

The basic obstruction is the square-root barrier.  If a set $A\subseteq\F_p^\times$ satisfies $A/A=\F_p^\times$, then the elementary inequality $|A|^2\ge p-1$ gives only $|A|\ge\sqrt{p-1}$.  Many naturally generated sets come with such a quotient or product identity, but an exponent larger than $1/2$ requires additional information.  The central theme here is that the needed extra information can be supplied by many \emph{internal transitions}: maps $T$ for which both an input and its output lie in the same value set.

A motivating example is the factorial-residue set
\[
  A_p=\{n!\bmod p:1\le n<p\}\subseteq\F_p^\times.
\]
Erd\H{o}s and Graham asked whether
\[
  |A_p|\sim(1-\mathrm e^{-1})p
  \qquad (p\to\infty);
\]
see \cite{ErdosGraham1980} and the formulation as Erd\H{o}s Problem~\#478 in \cite{Bloom478}.  The identity $n!/(n-1)!=n$, together with $0!=1=1!$, gives $A_p/A_p=\F_p^\times$ and hence the elementary square-root lower bound.  Grebennikov, Sagdeev, Semchankau and Vasilevskii improved its leading constant \cite{GrebennikovEtAl2024}; earlier work on the value set and average distribution appears in \cite{BanksEtAl2005,KlurmanMunsch2017}.  The exponent was subsequently improved to
\begin{equation}\label{eq:factorial-known}
  |A_p|\gg p^{8/15}
\end{equation}
by combining the identity
\begin{equation}\label{eq:factorial-identity-intro}
  (n+2)!=(n+1)!+\frac{((n+1)!)^2}{n!}
\end{equation}
with the Cartesian-product point-line incidence theorem of Stevens and de Zeeuw \cite{HuFactorials2026,StevensDeZeeuw2017}.  The important feature is not merely that \eqref{eq:factorial-identity-intro} is fractional-linear.  When two such transitions are paired at a common input, their relative transition is an affine map, and every nonidentity affine map has only boundedly many parameter representations.

The formally related self-power set
\[
  S_p=\{x^x\bmod p:1\le x<p\}
\]
illustrates the opposite situation.  Crocker proved a square-root-scale estimate \cite{Crocker1969}.  Bourgain and Shparlinski's work on consecutive modular roots supplied an important large-order method \cite{BourgainShparlinski2008}; Balog, Broughan and Shparlinski obtained uniform fibre and collision estimates \cite{BalogBroughanShparlinski2011}, and Cilleruelo and Garaev sharpened several individual-fibre bounds \cite{CillerueloGaraev2016}.  These results do not presently yield $|S_p|\gg p^{1/2+\delta}$ for a fixed $\delta>0$.  From the present viewpoint, the missing ingredient is a fixed-degree transition family whose quotients have low multiplicity and escape the proper subgroups of a small transformation group.

\subsection{Two transition-quotient mechanisms}

This paper isolates two regimes in which the preceding strategy closes.

In the first regime one has reciprocal-affine maps
\begin{equation}\label{eq:reciprocal-affine-intro}
  T_\theta(x)=b_\theta+\frac{c_\theta}{x},
  \qquad c_\theta\ne0.
\end{equation}
Writing $J(x)=1/x$ and $L_\theta(t)=c_\theta t+b_\theta$, one has $T_\theta=L_\theta\circ J$.  Consequently,
\[
  T_\eta\circ T_\theta^{-1}=L_\eta\circ L_\theta^{-1}\in\Aff(1,\F_p).
\]
After Cauchy--Schwarz, the second moment becomes an incidence count between $A\times A$ and affine lines.  A bounded quotient multiplicity then allows the Stevens--de Zeeuw theorem to produce the explicit exponent $8/15$.

In the second regime the consecutive ratios satisfy
\begin{equation}\label{eq:mobius-ratio-intro}
  \rho_{n+1}=\varphi(\rho_n),
  \qquad
  \varphi(t)=\frac{\alpha t+\beta}{\gamma t+\delta},
\end{equation}
with all four coefficients and the determinant nonzero.  If $a=u_{n+1}$ and $x=u_n$, then
\[
  u_{n+2}=a\varphi(a/x)=:T_a(x).
\]
The relative transformations $T_bT_a^{-1}$ lie in $\PSL_2(\F_p)$.  Their parameter map is injective off the diagonal, and Dickson's classification shows that the resulting set has only $O(|A|)$ elements in every proper subgroup coset.  Bourgain's modular Szemer\'edi--Trotter theorem for hyperbolas then yields a fixed, though non-explicit, power saving over the square-root exponent.

The two mechanisms may be summarized as follows.

\begin{center}
\small
\begin{tabular}{@{}lll@{}}
\toprule
Sequence or family & quotient geometry & lower bound \\
\midrule
factorials and arithmetic Pochhammer products & affine lines & $M^{8/15}$ \\
Gaussian $q$-factorials & affine lines & $M^{8/15}$ \\
derangements and ordered subsets & affine lines & $M^{8/15}$ \\
fixed binomial rows & subgroup-escaping $\PSL_2$ family & $M^{1/2+\eta}$ \\
Catalan and central binomial blocks & subgroup-escaping $\PSL_2$ family & $M^{1/2+\eta}$ \\
\bottomrule
\end{tabular}
\end{center}

Here $M$ denotes the number of distinct internal transitions available in the relevant block.

\subsection{Main results}

We first state the affine transition theorem.  Its quotient multiplicity will be defined precisely in \cref{sec:affine}.

\begin{theorem}\label{thm:affine-intro}
Let $A\subseteq\F_p$, let $\Theta$ be a finite parameter set with $|\Theta|\le |A|$, and let $(T_\theta)_{\theta\in\Theta}$ be pairwise distinct maps of the form \eqref{eq:reciprocal-affine-intro}.  Suppose that at least $M$ ordered pairs $(x,\theta)$ satisfy
\[
  x\in A\cap\F_p^\times,
  \qquad
  T_\theta(x)\in A.
\]
If every nonidentity quotient $T_\eta T_\theta^{-1}$ has at most $\mu$ parameter representations, where $\mu\ge1$, then
\begin{equation}\label{eq:affine-intro-bound}
  |A|\gg \min\{M,p\}^{8/15}\mu^{-4/15}.
\end{equation}
\end{theorem}

A particularly convenient consequence treats quadratic reciprocal transitions; see \cref{cor:quadratic-reciprocal}.  It leads to the following applications.

\begin{theorem}\label{thm:affine-applications-intro}
The following estimates hold.
\begin{enumerate}[label=(\roman*)]
  \item Let $r,d$ be fixed positive integers, and put
  \[
    U_n(r,d)=\prod_{j=0}^{n-1}(r+jd),
    \qquad U_0(r,d)=1.
  \]
  For all sufficiently large primes $p$,
  \[
    \left|\{U_n(r,d)\bmod p:0\le n,\ r+(n-1)d<p\}\right|
    \gg_d p^{8/15}.
  \]
  In particular this contains the factorial and double-factorial sequences.

  \item Let $g\in\F_p^\times$ have multiplicative order $L\ge2$, with $g\ne1$, and define
  \[
    [n]_g=\frac{1-g^n}{1-g},
    \qquad
    [n]_g!=\prod_{j=1}^n[j]_g.
  \]
  Then
  \begin{equation}\label{eq:qfactorial-intro}
    \left|\{[n]_g!:0\le n\le L-1\}\right|\gg L^{8/15}.
  \end{equation}
  In particular, if $g$ is primitive, the right-hand side is $\gg p^{8/15}$.

  \item If $D_n$ is the derangement sequence, then
  \begin{equation}\label{eq:derangement-intro}
    \left|\{D_0,D_1,\ldots,D_{p-1}\}\bmod p\right|\gg p^{8/15}.
  \end{equation}

  \item Let
  \[
    R_n=\sum_{k=0}^n\frac{n!}{(n-k)!},
  \]
  the number of ordered subsets of an $n$-element set.  Then
  \begin{equation}\label{eq:arrangement-intro}
    \left|\{R_0,R_1,\ldots,R_{p-1}\}\bmod p\right|\gg p^{8/15}.
  \end{equation}
\end{enumerate}
All implied constants in (ii)--(iv) are absolute.
\end{theorem}

We next state the M\"obius transition theorem.  Let
\begin{equation}\label{eq:phi-intro}
  \varphi(t)=\frac{\alpha t+\beta}{\gamma t+\delta},
  \qquad
  \alpha,\beta,\gamma,\delta\in\F_p^\times,
  \qquad
  \Delta:=\alpha\delta-\beta\gamma\ne0,
\end{equation}
and define
\begin{equation}\label{eq:T-intro}
  T_a(x)=a\varphi(a/x)
  =a\frac{\alpha a+\beta x}{\gamma a+\delta x}.
\end{equation}
For $V\subseteq\F_p^\times$, put
\begin{equation}\label{eq:transition-count-intro}
  \cN_T(V)
  =\#\{(x,a)\in V^2:\gamma a+\delta x\ne0,\ T_a(x)\in V\}.
\end{equation}

\begin{theorem}\label{thm:abstract-intro}
There are absolute constants $c,\eta>0$ such that, for every prime $p$, every quadruple satisfying \eqref{eq:phi-intro}, and every $V\subseteq\F_p^\times$,
\begin{equation}\label{eq:abstract-intro}
  |V|\ge c\min\{\cN_T(V),p\}^{1/2+\eta}.
\end{equation}
\end{theorem}

The exponent $\eta$ comes from expansion in $\SL_2(\F_p)$ and is not made explicit.  The applications below are all local: they concern blocks of length at most $p$.

\begin{theorem}\label{thm:binomial-intro}
There are absolute constants $c,\eta>0$ such that, for every prime $p$ and every integer $N$ with $1\le N\le p-3$,
\begin{equation}\label{eq:binomial-intro}
  \left|\left\{\binom Nn\bmod p:0\le n\le N\right\}\right|
  \ge c(N+1)^{1/2+\eta}.
\end{equation}
For the two remaining rows below $p$,
\begin{align}
  \left|\left\{\binom{p-2}{n}\bmod p:0\le n\le p-2\right\}\right|
  &=\frac{p-1}{2},\label{eq:pminus2-intro}\\
  \left|\left\{\binom{p-1}{n}\bmod p:0\le n\le p-1\right\}\right|
  &=2.\label{eq:pminus1-intro}
\end{align}
\end{theorem}

\begin{theorem}\label{thm:catalan-intro}
There are absolute constants $c,\eta>0$ such that, for every prime $p\ge7$,
\begin{align}
  \left|\left\{\frac1{n+1}\binom{2n}{n}\bmod p:
  0\le n\le\frac{p-1}{2}\right\}\right|
  &\ge cp^{1/2+\eta},\label{eq:catalan-intro}\\
  \left|\left\{\binom{2n}{n}\bmod p:
  0\le n\le\frac{p-1}{2}\right\}\right|
  &\ge cp^{1/2+\eta}.\label{eq:central-intro}
\end{align}
\end{theorem}

\subsection{Background for the applications}

Arithmetic Pochhammer products and Gaussian $q$-factorials are classical hypergeometric objects; see, for example, \cite{GasperRahman2004}.  The ordinary factorial is the arithmetic progression case $r=d=1$, while $[n]_g!$ is the finite-field specialization of a $q$-shifted factorial after the standard normalization by $(1-g)^n$.

Derangements are among the most classical permutation sequences.  Their congruences, periodicity, valuations and prime divisors have been studied in \cite{SunZagier2011,Miska2016}; standard combinatorial background can be found in \cite{Stanley2012}.  The ordered-subset numbers $R_n$ count injective words of all possible lengths on an $n$-element alphabet.  Their simple inhomogeneous recurrence makes them a natural companion to derangements in the reciprocal-quadratic framework.

Binomial coefficients modulo primes have a long arithmetic and digital theory.  The distribution of entries in Pascal's triangle was studied, among other places, by Garfield and Wilf \cite{GarfieldWilf1992}, Barbolosi and Grabner \cite{BarbolosiGrabner1996}, and Barat and Grabner \cite{BaratGrabner2001}.  For a fixed row, Mattarei investigated linear recurrences modulo a prime \cite{Mattarei2008}.  The statistic considered here is the number of distinct residues within a single row $N<p$.

Catalan numbers and central binomial coefficients have been studied through congruences, character sums, solution counts and automata.  Garaev, Luca and Shparlinski proved character-sum estimates and showed that sufficiently long initial segments cover all residue classes \cite{GaraevLucaShparlinski2006,GaraevLucaShparlinski2007}.  Burns proved that the full Catalan sequence assumes every residue modulo every prime $p\ge5$, indeed infinitely often \cite{Burns2017}.  The blocks in \cref{thm:catalan-intro} have length only about $p/2$, so eventual surjectivity does not directly address them.

The incidence-theoretic background also has two branches.  Sum-product methods over prime fields, beginning with Bourgain, Katz and Tao \cite{BourgainKatzTao2004}, led to strong incidence bounds and to Helfgott's growth theorem in $\SL_2(\F_p)$ \cite{Helfgott2008}.  Bourgain and Gamburd established uniform expansion for Cayley graphs of $\SL_2(\F_p)$ \cite{BourgainGamburd2008}.  Stevens and de Zeeuw proved the Cartesian-product point-line estimate used in the affine branch \cite{StevensDeZeeuw2017}.  Bourgain's modular Szemer\'edi--Trotter theorem treats M\"obius hyperbolae \cite{Bourgain2012}; later developments include \cite{Shkredov2021,RudnevWheeler2022,WarrenWheeler2023,JingZou2024}.

\subsection{Organization of the paper}

\Cref{sec:pointline} records the Cartesian-product incidence theorem.  \Cref{sec:affine} proves the reciprocal-affine transition theorem and gives two easy structural criteria for bounded quotient multiplicity.  The applications to Pochhammer products, $q$-factorials, derangements and ordered subsets appear in \cref{sec:affine-applications}.  The second half begins with Bourgain's incidence theorem in \cref{sec:bourgain}.  The normalized $\SL_2$ quotient matrices and their injectivity are established in \cref{sec:transitions}; subgroup escape is proved in \cref{sec:escape}; and \cref{sec:abstract-proof} proves \cref{thm:abstract-intro}.  M\"obius ratio dynamics and the binomial, Catalan and central-binomial applications are treated in \cref{sec:ratio,sec:binomial,sec:catalan}.

Throughout, all implied constants are absolute unless a dependence is indicated.

\section{A Cartesian-product point-line theorem}\label{sec:pointline}

For a point set $\cP\subseteq\F_p^2$ and a finite set $\cL$ of affine lines, write
\[
  I(\cP,\cL)
  =\#\{(q,\ell)\in\cP\times\cL:q\in\ell\}.
\]
We use the following form of the Cartesian-product incidence theorem of Stevens and de Zeeuw \cite[Theorem~4]{StevensDeZeeuw2017}.

\begin{theorem}[Stevens--de Zeeuw]\label{thm:cartesian-incidence}
There are absolute constants $c_0,C_0>0$ with the following property.  Let $X,Y\subseteq\F_p$, with $|X|=x\le y=|Y|$, and let $\cL$ be a set of $N$ distinct affine lines.  Suppose that
\[
  xy^2\le N^3
  \qquad\text{and}\qquad
  xN\le c_0p^2.
\]
Then
\begin{equation}\label{eq:cartesian-incidence}
  I(X\times Y,\cL)
  \le C_0\bigl(x^{3/4}y^{1/2}N^{3/4}+N\bigr).
\end{equation}
\end{theorem}

The application below always has $X=Y=A$ and, after padding the line set if necessary, $N=|A|^2$.  The characteristic condition is then $|A|^3\ll p^2$, and \eqref{eq:cartesian-incidence} becomes
\begin{equation}\label{eq:cartesian-special}
  I(A\times A,\cL)\ll |A|^{11/4}.
\end{equation}
If the characteristic condition fails, the set already has size $\gg p^{2/3}$, which is stronger than every affine-branch conclusion needed here.

\section{Reciprocal-affine transitions}\label{sec:affine}

Let $A\subseteq\F_p$, put $m=|A|$, and let $\Theta$ be a finite parameter set.  For each $\theta\in\Theta$, fix $b_\theta\in\F_p$ and $c_\theta\in\F_p^\times$, and define
\begin{equation}\label{eq:affine-T}
  T_\theta(x)=b_\theta+\frac{c_\theta}{x}
  \qquad (x\in\F_p^\times).
\end{equation}
We assume that the maps $T_\theta$ are pairwise distinct.  Their inverses are
\[
  T_\theta^{-1}(u)=\frac{c_\theta}{u-b_\theta},
\]
and hence
\begin{equation}\label{eq:affine-quotient}
  T_\eta\bigl(T_\theta^{-1}(u)\bigr)
  =\frac{c_\eta}{c_\theta}u
   +b_\eta-\frac{c_\eta}{c_\theta}b_\theta.
\end{equation}
Thus every relative transition is an affine line.

Define the nonidentity quotient multiplicity by
\begin{equation}\label{eq:affine-multiplicity}
  \mu(\Theta)
  =\max\left\{1,
  \max_{\ell\ne\id}
  \#\{(\theta,\eta)\in\Theta^2:
  T_\eta T_\theta^{-1}=\ell\}\right\}.
\end{equation}
Pairwise distinctness ensures that the identity quotient occurs only when $\theta=\eta$.

For a set $\cE\subseteq(A\cap\F_p^\times)\times\Theta$, call $(x,\theta)\in\cE$ an internal transition if $T_\theta(x)\in A$.

\begin{theorem}\label{thm:affine}
Assume $|\Theta|\le m$, and suppose that there are at least $M$ internal transitions.  Then
\begin{equation}\label{eq:affine-main}
  m\gg \min\{M,p\}^{8/15}\mu(\Theta)^{-4/15}.
\end{equation}
\end{theorem}

\begin{proof}
Put $L=\min\{M,p\}$ and choose exactly $L$ distinct internal transitions.  For $x\in A\cap\F_p^\times$, let
\[
  R(x)=\#\{\theta:(x,\theta)\text{ was selected}\}.
\]
Then $\sum_xR(x)=L$, so Cauchy--Schwarz gives
\begin{equation}\label{eq:affine-CS}
  L^2
  \le m\sum_xR(x)^2
  \le m\sum_{\theta,\eta\in\Theta}M(\theta,\eta),
\end{equation}
where
\[
  M(\theta,\eta)
  =\#\{x\in A\cap\F_p^\times:
       T_\theta(x),T_\eta(x)\in A\}.
\]
The diagonal pairs contribute at most $|\Theta|m\le m^2$.

For $\theta\ne\eta$, put $u=T_\theta(x)$ and $v=T_\eta(x)$.  The substitution is injective in $x$, and \eqref{eq:affine-quotient} places $(u,v)\in A\times A$ on the affine line representing $T_\eta T_\theta^{-1}$.  If $\cL_\Theta$ is the set of distinct nonidentity lines that arise, then
\begin{equation}\label{eq:affine-second-moment}
  \sum_{\theta,\eta\in\Theta}M(\theta,\eta)
  \le m^2+\mu(\Theta)I(A\times A,\cL_\Theta).
\end{equation}
Moreover, $|\cL_\Theta|\le|\Theta|^2\le m^2$.

Let $c_0$ be the constant in \cref{thm:cartesian-incidence}.  If $m^3>c_0p^2$, then
\[
  m\gg p^{2/3}\ge L^{2/3},
\]
which is stronger than \eqref{eq:affine-main}.  Otherwise, enlarge $\cL_\Theta$ to a set of exactly $m^2$ nonvertical affine lines and apply \cref{thm:cartesian-incidence}.  By \eqref{eq:cartesian-special},
\[
  I(A\times A,\cL_\Theta)\ll m^{11/4}.
\]
Combining this estimate with \eqref{eq:affine-CS} and \eqref{eq:affine-second-moment} yields
\[
  L^2\ll m\bigl(m^2+\mu(\Theta)m^{11/4}\bigr)
  \ll \mu(\Theta)m^{15/4}.
\]
Rearranging proves the theorem.
\end{proof}

The theorem can be viewed inside the affine group.  Let $J(x)=1/x$ and $L_\theta(t)=c_\theta t+b_\theta$.  Then $T_\theta=L_\theta J$, and the quotient multiplicity in \eqref{eq:affine-multiplicity} is exactly the multiplicative energy of the parameter set $(L_\theta)$ at a single quotient.

\begin{proposition}\label{prop:affine-curve}
Let $\Gamma\subseteq\Aff(1,\overline{\F}_p)$ be an irreducible algebraic curve of degree at most $d$.  Assume that no nonidentity element $g\in\Aff(1,\overline{\F}_p)$ satisfies $g\Gamma=\Gamma$.  If the distinct affine maps $L_\theta$ all lie on $\Gamma$, then
\begin{equation}\label{eq:affine-curve-multiplicity}
  \mu(\Theta)\le\max\{1,d^2\}.
\end{equation}
\end{proposition}

\begin{proof}
Fix a nonidentity affine map $g$.  The relation $L_\eta L_\theta^{-1}=g$ is equivalent to $L_\eta=gL_\theta$.  Hence every representation gives a point of
\[
  \Gamma\cap g^{-1}\Gamma.
\]
By assumption these are distinct irreducible curves.  B\'ezout's theorem bounds their intersection by $d^2$, counted over the algebraic closure.  The point $L_\theta$ determines $L_\eta$, so the same bound holds for the number of ordered parameter pairs.
\end{proof}

A second criterion is elementary and is the one used for derangements and ordered subsets.

\begin{corollary}\label{cor:quadratic-reciprocal}
Let $A\subseteq\F_p$, let $B\subseteq A$, and let $F:B\to\F_p$ be injective.  Let
\[
  Q(t)=ct^2+dt+e,
  \qquad c\ne0,
\]
and suppose that $Q(F(a))\ne0$ for every $a\in B$.  Define
\begin{equation}\label{eq:quadratic-transition}
  T_a(x)=F(a)+\frac{Q(F(a))}{x}.
\end{equation}
If at least $M$ pairs $(x,a)\in(A\cap\F_p^\times)\times B$ satisfy $T_a(x)\in A$, then
\begin{equation}\label{eq:quadratic-bound}
  |A|\gg\min\{M,p\}^{8/15}.
\end{equation}
\end{corollary}

\begin{proof}
Put $t=F(a)$ and $s=F(b)$.  If the relative transition $T_bT_a^{-1}$ is the affine map $u\mapsto\lambda u+\tau$, then
\begin{equation}\label{eq:quadratic-conditions}
  s=\lambda t+\tau,
  \qquad
  Q(\lambda t+\tau)=\lambda Q(t).
\end{equation}
The second relation is a polynomial equation of degree at most two in $t$.  It cannot vanish identically unless $\lambda=1$ and $\tau=0$: the coefficients of $t^2$ and $t$ are $c\lambda(\lambda-1)$ and $2c\lambda\tau$, respectively.  Thus every nonidentity quotient has at most two possible values of $t$, and the injectivity of $F$ gives at most two parameter pairs.  Apply \cref{thm:affine} with $\mu\le2$.
\end{proof}

\begin{corollary}\label{cor:affine-ratio}
Let $u_0,u_1,\ldots,u_K\in\F_p^\times$, and put
\[
  \rho_n=\frac{u_{n+1}}{u_n}.
\]
Assume that
\begin{equation}\label{eq:affine-ratio-recurrence}
  \rho_{n+1}=\alpha\rho_n+\beta,
  \qquad \alpha\beta\ne0,
\end{equation}
for $0\le n\le K-2$.  If $L$ of the ordered pairs $(u_n,u_{n+1})$ are distinct, then
\begin{equation}\label{eq:affine-ratio-bound}
  |\{u_0,u_1,\ldots,u_K\}|
  \gg\min\{L,p\}^{8/15}.
\end{equation}
\end{corollary}

\begin{proof}
Equation \eqref{eq:affine-ratio-recurrence} gives
\[
  u_{n+2}
  =\beta u_{n+1}+\alpha\frac{u_{n+1}^2}{u_n}.
\]
This is \cref{cor:quadratic-reciprocal} with
\[
  F(a)=\beta a,
  \qquad
  Q(t)=\frac{\alpha}{\beta^2}t^2.
\]
\end{proof}

\section{Applications of the affine theorem}\label{sec:affine-applications}

\subsection{Arithmetic Pochhammer products}

Let $1\le r<p$ and $d\ge1$, and define
\begin{equation}\label{eq:pochhammer-def}
  U_0(r,d)=1,
  \qquad
  U_n(r,d)=\prod_{j=0}^{n-1}(r+jd).
\end{equation}
Let
\begin{equation}\label{eq:J-rd}
  J=1+\left\lfloor\frac{p-1-r}{d}\right\rfloor.
\end{equation}
Then every factor occurring in $U_0,\ldots,U_J$ lies in $\{1,\ldots,p-1\}$.

\begin{proposition}\label{prop:pochhammer}
With the notation above,
\begin{equation}\label{eq:pochhammer-bound}
  \left|\{U_0(r,d),\ldots,U_J(r,d)\}\bmod p\right|
  \gg (J+1)^{8/15}.
\end{equation}
Consequently, for fixed $r,d$ and $p\to\infty$,
\[
  \left|\{U_n(r,d)\bmod p:r+(n-1)d<p\}\right|
  \gg_d p^{8/15}.
\]
\end{proposition}

\begin{proof}
The adjacent ratios are
\[
  \rho_n=\frac{U_{n+1}}{U_n}=r+nd,
  \qquad 0\le n\le J-1,
\]
and they satisfy $\rho_{n+1}=\rho_n+d$.  They are pairwise distinct modulo $p$ because they are distinct integers in $\{1,\ldots,p-1\}$.  Hence the adjacent pairs $(U_n,U_{n+1})$, $0\le n\le J-2$, are distinct.  Apply \cref{cor:affine-ratio} with $(\alpha,\beta)=(1,d)$.  The cases of bounded $J$ are absorbed into the implied constant.
\end{proof}

Taking $(r,d)=(1,1)$ recovers \eqref{eq:factorial-known}.  The choices $(r,d)=(1,2)$ and $(2,2)$ give the odd and even double factorials.  More generally, fixed-step multifactorial sequences are arithmetic Pochhammer products.

\subsection{Gaussian $q$-factorials}

Let $g\in\F_p^\times$ with $g\ne1$, and write $L=\ord(g)$.  For $0\le n<L$, define
\begin{equation}\label{eq:qinteger-def}
  [n]_g=\frac{1-g^n}{1-g},
  \qquad
  [0]_g!=1,
  \qquad
  [n]_g!=\prod_{j=1}^n[j]_g.
\end{equation}
These are the standard Gaussian integers and factorials specialized at $g$; see \cite{GasperRahman2004}.  Since $g^j\ne1$ for $1\le j<L$, all terms in \eqref{eq:qinteger-def} are nonzero.

\begin{proposition}\label{prop:qfactorial}
For every $g\ne1$ of order $L$,
\begin{equation}\label{eq:qfactorial-bound}
  \left|\{[0]_g!,[1]_g!,\ldots,[L-1]_g!\}\right|
  \gg L^{8/15}.
\end{equation}
\end{proposition}

\begin{proof}
Put $u_n=[n]_g!$.  Its adjacent ratio is $\rho_n=[n+1]_g$.  The identity
\[
  [n+2]_g=1+g[n+1]_g
\]
gives
\[
  \rho_{n+1}=g\rho_n+1.
\]
Moreover, $[1]_g,\ldots,[L-1]_g$ are pairwise distinct because $g,g^2,\ldots,g^{L-1}$ are distinct.  Hence the adjacent pairs $(u_n,u_{n+1})$, $0\le n\le L-3$, are distinct.  Apply \cref{cor:affine-ratio}; the finitely many cases $L<4$ are harmless.
\end{proof}

If $g$ is a primitive root, then $L=p-1$ and \eqref{eq:qfactorial-bound} is $\gg p^{8/15}$.

\subsection{A zero-spacing observation}

The next two sequences can vanish modulo $p$, while \cref{cor:quadratic-reciprocal} requires a nonzero input and a nonzero reciprocal coefficient.  The following elementary observation supplies enough admissible adjacent pairs.

\begin{lemma}\label{lem:zero-spacing}
Let $w_0,w_1,\ldots,w_M\in\F_p$, and suppose that any two zero terms have indices differing by at least three.  Then
\begin{equation}\label{eq:good-adjacent-pairs}
  \#\{0\le n<M:w_nw_{n+1}\ne0\}
  \ge\frac{M-6}{3}.
\end{equation}
\end{lemma}

\begin{proof}
There are at most $(M+3)/3$ zero terms.  Each zero can spoil at most two adjacent pairs, so the number of good adjacent pairs is at least
\[
  M-2\frac{M+3}{3}=\frac{M-6}{3}.
\]
\end{proof}

\subsection{Derangements}

Let $D_n$ denote the number of derangements of an $n$-element set.  We use the standard recurrence
\begin{equation}\label{eq:derangement-recurrence}
  D_0=1,
  \qquad
  D_{n+1}=(n+1)D_n+(-1)^{n+1}.
\end{equation}
Arithmetic aspects of this sequence are studied in \cite{SunZagier2011,Miska2016}.

\begin{proposition}\label{prop:derangements}
For every prime $p$,
\begin{equation}\label{eq:derangement-bound}
  \left|\{D_0,D_1,\ldots,D_{p-1}\}\bmod p\right|
  \gg p^{8/15}.
\end{equation}
\end{proposition}

\begin{proof}
We may assume that $p$ is odd and sufficiently large.  If $D_j=0$ in $\F_p$, then \eqref{eq:derangement-recurrence} gives
\[
  D_{j+1}=(-1)^{j+1},
  \qquad
  D_{j+2}=(j+1)(-1)^{j+1}.
\]
For $0\le j\le p-2$, the latter is nonzero whenever its index lies in the range under consideration.  Thus the zero terms among $D_0,\ldots,D_{p-2}$ have index gaps at least three.  By \cref{lem:zero-spacing}, there are $\gg p$ indices $0\le n\le p-3$ such that
\[
  D_nD_{n+1}\ne0.
\]
At least half of these indices have the same parity.  Fix that parity and put
\[
  \varepsilon=(-1)^{n+1},
  \qquad
  x=D_n,
  \qquad
  a=D_{n+1}.
\]
Then
\[
  a-\varepsilon=(n+1)x.
\]
Using the recurrence once more,
\begin{align*}
  D_{n+2}
  &=(n+2)a-\varepsilon\\
  &=(a-\varepsilon)+\frac{a(a-\varepsilon)}{x}.
\end{align*}
Therefore the transition has the form
\[
  T_a(x)=F_\varepsilon(a)+\frac{Q_\varepsilon(F_\varepsilon(a))}{x},
  \qquad
  F_\varepsilon(a)=a-\varepsilon,
  \qquad
  Q_\varepsilon(t)=t(t+\varepsilon).
\]
For every selected index, $F_\varepsilon(a)=(n+1)x\ne0$ and $F_\varepsilon(a)+\varepsilon=a\ne0$, so the reciprocal coefficient is nonzero.  The selected ordered pairs $(x,a)$ are distinct, because
\[
  n+1=\frac{a-\varepsilon}{x}
\]
recovers the index in $\F_p$.  Thus \cref{cor:quadratic-reciprocal} applies with $M\gg p$ and proves \eqref{eq:derangement-bound}.
\end{proof}

\subsection{Ordered subsets}

Let
\begin{equation}\label{eq:arrangement-def}
  R_n=\sum_{k=0}^n\frac{n!}{(n-k)!}
  =n!\sum_{j=0}^n\frac1{j!}.
\end{equation}
Thus $R_n$ counts all ordered subsets, or injective words, formed from an $n$-element set; see \cite{Stanley2012}.  It satisfies
\begin{equation}\label{eq:arrangement-recurrence}
  R_0=1,
  \qquad
  R_{n+1}=(n+1)R_n+1.
\end{equation}

\begin{proposition}\label{prop:arrangements}
For every prime $p$,
\begin{equation}\label{eq:arrangement-bound}
  \left|\{R_0,R_1,\ldots,R_{p-1}\}\bmod p\right|
  \gg p^{8/15}.
\end{equation}
\end{proposition}

\begin{proof}
Again assume that $p$ is odd and sufficiently large.  If $R_j=0$, then
\[
  R_{j+1}=1,
  \qquad
  R_{j+2}=j+3.
\]
It follows that the zero terms among $R_0,\ldots,R_{p-2}$ have index gaps at least three.  Hence there are $\gg p$ indices $0\le n\le p-3$ for which $R_nR_{n+1}\ne0$.

For such an index, put $x=R_n$ and $a=R_{n+1}$.  Since
\[
  a-1=(n+1)x,
\]
we have
\begin{align*}
  R_{n+2}
  &=(n+2)a+1\\
  &=(a+1)+\frac{a(a-1)}{x}.
\end{align*}
This is \eqref{eq:quadratic-transition} with
\[
  F(a)=a+1,
  \qquad
  Q(t)=(t-1)(t-2).
\]
The reciprocal coefficient is nonzero: $t-1=a\ne0$, while $t-2=a-1=(n+1)x\ne0$.  The selected adjacent pairs are distinct because
\[
  n+1=\frac{a-1}{x}.
\]
Apply \cref{cor:quadratic-reciprocal} with $M\gg p$.
\end{proof}

\section{Bourgain's incidence theorem}\label{sec:bourgain}

We recall the projective formulation needed in the proof.  The projective line is
\[
  \Pone(\F_p)=\F_p\cup\{\infty\}.
\]
A matrix
\[
  g=\begin{pmatrix}a&b\\c&d\end{pmatrix}\in\GL_2(\F_p)
\]
acts by
\[
  g\cdot x=\frac{ax+b}{cx+d}
\]
with the usual projective conventions.  Multiplying $g$ by a nonzero scalar does not change the action.

Associated with $g$ is the affine curve
\begin{equation}\label{eq:Gamma-g}
  \Gamma_g
  =\{(x,y)\in\F_p^2:cxy-ax+dy-b=0\}.
\end{equation}
If $cx+d\ne0$, the equation is equivalent to $y=g\cdot x$.  If $cx+d=0$, the determinant condition shows that the equation has no affine point with that first coordinate.  Hence \eqref{eq:Gamma-g} records exactly the finite-to-finite part of the projective action.

For $A\subseteq\F_p$ and $S\subseteq\SL_2(\F_p)$, define
\begin{equation}\label{eq:incidence-def}
  I(A,S)
  =\#\{(x,y,g)\in A\times A\times S:(x,y)\in\Gamma_g\}.
\end{equation}

Bourgain's theorem is based on the expansion theory of $\SL_2(\F_p)$ developed by Bourgain and Gamburd \cite{BourgainGamburd2008}, together with Helfgott's growth theorem \cite{Helfgott2008}.  We use the following form of \cite[Proposition~1]{Bourgain2012}.

\begin{theorem}[Bourgain]\label{thm:bourgain}
For every $\varepsilon>0$ and $r>1$, there is a number $\sigma=\sigma(\varepsilon,r)>0$ such that the following holds for all sufficiently large primes $p$.  Let $A\subseteq\F_p$ and $S\subseteq\SL_2(\F_p)$ satisfy
\begin{align}
  1\ll |A|&<p^{1-\varepsilon},\label{eq:bourgain-size-A}\\
  \log|A|&<r\log|S|,\label{eq:bourgain-size-S}\\
  |S\cap gK|&<|S|^{1-\varepsilon}\label{eq:bourgain-nonconc}
\end{align}
for every proper subgroup $K<\SL_2(\F_p)$ and every $g\in\SL_2(\F_p)$.  Then
\begin{equation}\label{eq:bourgain-conclusion}
  I(A,S)\ll |A|^{1-\sigma}|S|.
\end{equation}
\end{theorem}

The hypothesis \eqref{eq:bourgain-nonconc} is indispensable.  If $S$ is concentrated in a coset of a point stabilizer or a torus normalizer, the transformations may preserve a small projective configuration and produce many incidences.  Our main algebraic task is therefore not only to produce many distinct quotient transformations, but also to show that they avoid every such coset.

Later quantitative incidence theorems for M\"obius transformations and modular hyperbolae were developed by Shkredov, Rudnev--Wheeler and Warren--Wheeler \cite{Shkredov2021,RudnevWheeler2022,WarrenWheeler2023}.  Jing and Zou subsequently placed Bourgain's result in a broader group-action Szemer\'edi--Trotter framework \cite{JingZou2024}.  The qualitative power saving in \cref{thm:bourgain} is especially well suited to the present family because its subgroup nonconcentration can be proved uniformly.

\section{The M\"obius transition family}\label{sec:transitions}

Fix
\begin{equation}\label{eq:coeff-assumptions}
  \alpha,\beta,\gamma,\delta\in\F_p^\times,
  \qquad
  \Delta=\alpha\delta-\beta\gamma\ne0.
\end{equation}
For $a\in\F_p^\times$, define
\begin{equation}\label{eq:T-def}
  T_a(x)=a\frac{\alpha a+\beta x}{\gamma a+\delta x}.
\end{equation}
As a projective transformation, $T_a$ is represented by
\begin{equation}\label{eq:G-a}
  G_a=
  \begin{pmatrix}
    \beta a&\alpha a^2\\
    \delta&\gamma a
  \end{pmatrix}.
\end{equation}
Indeed,
\[
  \det G_a=(\beta\gamma-\alpha\delta)a^2=-\Delta a^2\ne0.
\]

For $a,b\in\F_p^\times$, put
\begin{equation}\label{eq:H-def}
  H_{a,b}=\frac abG_bG_a^{-1}.
\end{equation}
Since
\[
  \det(G_bG_a^{-1})=\frac{b^2}{a^2},
\]
we have $H_{a,b}\in\SL_2(\F_p)$.  Its projective action is $T_b\circ T_a^{-1}$.

\begin{lemma}\label{lem:H-formula}
For all $a,b\in\F_p^\times$,
\begin{equation}\label{eq:H-formula}
  H_{a,b}
  =\frac1\Delta
  \begin{pmatrix}
    \alpha\delta\,\dfrac ba-\beta\gamma
    &\alpha\beta(a-b)\\[3mm]
    \gamma\delta\,\dfrac{b-a}{ab}
    &\alpha\delta\,\dfrac ab-\beta\gamma
  \end{pmatrix}.
\end{equation}
Moreover, the map
\[
  (a,b)\longmapsto H_{a,b}
\]
is injective on $\{(a,b)\in(\F_p^\times)^2:a\ne b\}$.
\end{lemma}

\begin{proof}
The displayed formula follows by multiplying \eqref{eq:G-a} and its inverse.  To prove injectivity, write $r=b/a$.  The upper-left entry of $H_{a,b}$ is
\[
  \frac{\alpha\delta r-\beta\gamma}{\Delta},
\]
so it determines $r$, because $\alpha\delta\ne0$.  Since $a\ne b$, one has $r\ne1$, and the upper-right entry
\[
  \frac{\alpha\beta a(1-r)}{\Delta}
\]
then determines $a$.  Finally $b=ra$.
\end{proof}

For a set $V\subseteq\F_p^\times$, define
\begin{equation}\label{eq:S-V}
  \cS_V=\{H_{a,b}:a,b\in V,\ a\ne b\}\subseteq\SL_2(\F_p).
\end{equation}
By \cref{lem:H-formula}, if $m=|V|$, then
\begin{equation}\label{eq:S-size-exact}
  |\cS_V|=m(m-1).
\end{equation}
The exact quadratic size will verify \eqref{eq:bourgain-size-S}.  The next section verifies the more delicate condition \eqref{eq:bourgain-nonconc}.

\section{Escape from proper subgroups}\label{sec:escape}

The first lemma bounds the number of quotient transformations that can send a prescribed projective point to another prescribed point.  We work over $\overline{\F}_p$ because nonsplit tori are diagonalized over $\F_{p^2}$.

\begin{lemma}\label{lem:evaluation}
Let $V\subseteq\F_p^\times$ and $m=|V|$.  For all
\[
  \xi,\eta\in\Pone(\overline{\F}_p),
\]
one has
\begin{equation}\label{eq:evaluation-bound}
  \#\{(a,b)\in V^2:(T_b\circ T_a^{-1})(\xi)=\eta\}\le2m.
\end{equation}
\end{lemma}

\begin{proof}
Fix $a\in V$ and put $z=T_a^{-1}(\xi)$.  We count $b\in V$ for which $T_b(z)=\eta$.

Suppose first that $z$ and $\eta$ are finite.  Clearing the denominator in \eqref{eq:T-def} gives
\begin{equation}\label{eq:b-quadratic}
  \alpha b^2+(\beta z-\eta\gamma)b-\eta\delta z=0.
\end{equation}
The leading coefficient is nonzero, so there are at most two possibilities for $b$.

If $z$ is finite and $\eta=\infty$, then $\gamma b+\delta z=0$, which has at most one solution.  If $z=\infty$ and $\eta$ is finite, then
\[
  T_b(\infty)=\frac{\beta}{\delta}b,
\]
so again there is at most one solution.  Finally, if $z=\eta=\infty$, there is no solution because $T_b(\infty)$ is finite.  Summing over the $m$ choices of $a$ proves the claim.
\end{proof}

We use the following standard consequence of Dickson's classification \cite{Dickson1901}; see also the subgroup discussion in \cite{Helfgott2008}.

\begin{proposition}\label{prop:dickson}
Let $p\ge5$ be prime.  Every proper subgroup of $\PSL_2(\F_p)$ is contained in one of the following:
\begin{enumerate}[label=(\roman*)]
  \item a Borel subgroup, which stabilizes a point of $\Pone(\F_p)$;
  \item the normalizer of a split or nonsplit torus, which stabilizes an unordered pair of points of $\Pone(\F_{p^2})$;
  \item an exceptional subgroup isomorphic to $A_4$, $S_4$ or $A_5$.
\end{enumerate}
\end{proposition}

A split torus fixes two points in $\Pone(\F_p)$.  A nonsplit torus fixes two Galois-conjugate points in $\Pone(\F_{p^2})\setminus\Pone(\F_p)$.  In either case, its normalizer preserves the corresponding unordered pair.

\begin{proposition}\label{prop:subgroup-escape}
There is an absolute constant $C$ such that, for every prime $p\ge5$, every $V\subseteq\F_p^\times$ with $m=|V|$, every proper subgroup $K<\SL_2(\F_p)$, and every $g\in\SL_2(\F_p)$,
\begin{equation}\label{eq:subgroup-escape}
  |\cS_V\cap gK|\le Cm+C.
\end{equation}
In particular, for all sufficiently large $m$,
\begin{equation}\label{eq:subgroup-escape-power}
  |\cS_V\cap gK|<|\cS_V|^{3/4}.
\end{equation}
\end{proposition}

\begin{proof}
Let
\[
  \pi:\SL_2(\F_p)\longrightarrow\PSL_2(\F_p)
\]
be the quotient map.  The group $\SL_2(\F_p)$ is perfect for $p\ge5$.  If $\pi(K)=\PSL_2(\F_p)$, then $K\{\pm I\}=\SL_2(\F_p)$ and hence $[\SL_2(\F_p):K]\le2$.  A proper subgroup would therefore have index two and be normal, producing a nontrivial abelian quotient of a perfect group.  This is impossible, so $\pi(K)$ is proper.

By \cref{prop:dickson}, $\pi(K)$ is contained in one of the three listed types.  We treat them separately.

If $\pi(K)$ is contained in a Borel subgroup stabilizing $\xi\in\Pone(\F_p)$, then every projective transformation in $\pi(gK)$ sends $\xi$ to the fixed point $\pi(g)\xi$.  By \cref{lem:evaluation}, at most $2m$ ordered parameter pairs $(a,b)\in V^2$ have this property.  Since $(a,b)\mapsto H_{a,b}$ is injective off the diagonal, this bounds $|\cS_V\cap gK|$ by $2m$.

Suppose next that $\pi(K)$ lies in a torus normalizer.  It preserves an unordered pair
\[
  \{\xi_1,\xi_2\}\subseteq\Pone(\F_{p^2}).
\]
Every element of the coset $\pi(gK)$ sends $\xi_1$ to one of $\pi(g)\xi_1$ and $\pi(g)\xi_2$.  Applying \cref{lem:evaluation} to the two possible target points gives at most $4m$ parameter pairs.

Finally, an exceptional projective subgroup has order at most $60$.  Its inverse image in $\SL_2(\F_p)$ has order at most $120$, so any of its cosets meets $\cS_V$ in at most $120$ elements.

This proves \eqref{eq:subgroup-escape}.  The last assertion follows from \eqref{eq:S-size-exact}, since $m(m-1)\asymp m^2$.
\end{proof}

\begin{remark}\label{rem:nonzero-coefficients}
The assumptions $\alpha\beta\gamma\delta\ne0$ have a structural role.  They ensure that the equation in \eqref{eq:b-quadratic} is genuinely quadratic and that the quotient family cannot have a large common fixed-point or fixed-pair degeneration.  Some vanishing-coefficient cases can be treated separately.  The factorial transition is an important affine degeneration, where a point-line theorem gives a stronger explicit exponent \cite{HuFactorials2026}.
\end{remark}

\section{The abstract value-set theorem}\label{sec:abstract-proof}

We now prove \cref{thm:abstract-intro}.  It is useful first to record the transition-count formulation precisely.

For $V\subseteq\F_p^\times$, let $\cN_T(V)$ be as in \eqref{eq:transition-count-intro}.  If $L\le\min\{\cN_T(V),p\}$, choose a set
\begin{equation}\label{eq:E-selected}
  \cE\subseteq\{(x,a)\in V^2:T_a(x)\in V\}
\end{equation}
of exactly $L$ distinct ordered pairs.  For $x\in V$, put
\begin{equation}\label{eq:R-E}
  R_{\cE}(x)=\#\{a\in V:(x,a)\in\cE\}.
\end{equation}
Then
\[
  \sum_{x\in V}R_{\cE}(x)=L.
\]
The inequality $L\le |V|^2$ already gives the square-root bound $|V|\ge L^{1/2}$.  The role of the second moment below is to replace this tautological estimate by a fixed power improvement.

\begin{proof}[Proof of \cref{thm:abstract-intro}]
Write $m=|V|$.  By Cauchy--Schwarz,
\begin{equation}\label{eq:CS-main}
  L^2
  \le m\sum_{x\in V}R_{\cE}(x)^2
  \le m\sum_{a,b\in V}M(a,b),
\end{equation}
where
\begin{equation}\label{eq:M-ab}
  M(a,b)
  =\#\{x\in V:T_a(x)\in V,\ T_b(x)\in V\}.
\end{equation}
The diagonal terms contribute at most $m^2$.

Suppose $a\ne b$.  Put $u=T_a(x)$.  Since $T_a$ is a projective automorphism, the substitution is injective, and
\[
  T_b(x)=(T_b\circ T_a^{-1})(u)=H_{a,b}\cdot u.
\]
Consequently,
\begin{equation}\label{eq:M-incidence}
  \sum_{\substack{a,b\in V\\a\ne b}}M(a,b)
  \le I(V,\cS_V).
\end{equation}

Fix $\varepsilon=1/4$ and $r=2$ in \cref{thm:bourgain}, and let $\sigma>0$ be the resulting saving.  We may assume $0<\sigma\le1$.

If $m\ge p^{3/4}$, then $L\le p$ gives
\begin{equation}\label{eq:large-m-case}
  m\ge L^{3/4}.
\end{equation}
Suppose therefore that $m<p^{3/4}$.  For sufficiently large $m$, \eqref{eq:S-size-exact} gives
\[
  \log m<2\log|\cS_V|,
\]
and \cref{prop:subgroup-escape} gives
\[
  |\cS_V\cap gK|<|\cS_V|^{3/4}
\]
for every proper subgroup $K$ and every coset.  Thus all hypotheses of \cref{thm:bourgain} are satisfied, and
\begin{equation}\label{eq:I-bound}
  I(V,\cS_V)
  \ll m^{1-\sigma}|\cS_V|
  \ll m^{3-\sigma}.
\end{equation}
Combining \eqref{eq:CS-main}, \eqref{eq:M-incidence} and \eqref{eq:I-bound}, we obtain
\begin{equation}\label{eq:L-bound}
  L^2\ll m(m^2+m^{3-\sigma})\ll m^{4-\sigma}.
\end{equation}
Hence
\begin{equation}\label{eq:m-from-L}
  m\gg L^{2/(4-\sigma)}.
\end{equation}
Set
\begin{equation}\label{eq:eta-def}
  \eta=\frac{2}{4-\sigma}-\frac12
  =\frac{\sigma}{2(4-\sigma)}>0.
\end{equation}
The exponent in \eqref{eq:large-m-case} is stronger than $1/2+\eta$.  The finitely many cases in which $p$ or $m$ is below the thresholds in Bourgain's theorem are absorbed by decreasing the absolute implied constant.  Taking $L=\min\{\cN_T(V),p\}$ proves \eqref{eq:abstract-intro}.
\end{proof}

\section{Sequences with M\"obius ratio dynamics}\label{sec:ratio}

The abstract theorem applies whenever a sequence supplies many distinct transitions of the form \eqref{eq:T-def}.

\begin{corollary}\label{cor:ratio-dynamics}
Let $u_0,u_1,\ldots,u_K\in\F_p^\times$, and put
\[
  \rho_n=\frac{u_{n+1}}{u_n}
  \qquad(0\le n<K).
\]
Assume that for $0\le n\le K-2$,
\begin{equation}\label{eq:rho-recurrence}
  \rho_{n+1}=\frac{\alpha\rho_n+\beta}{\gamma\rho_n+\delta},
\end{equation}
where \eqref{eq:coeff-assumptions} holds.  Let $L$ be the number of distinct ordered pairs
\[
  (u_n,u_{n+1}),
  \qquad 0\le n\le K-2.
\]
Then
\begin{equation}\label{eq:ratio-corollary}
  |\{u_0,u_1,\ldots,u_K\}|
  \gg \min\{L,p\}^{1/2+\eta},
\end{equation}
where $\eta>0$ is the absolute constant in \cref{thm:abstract-intro}.
\end{corollary}

\begin{proof}
Equation \eqref{eq:rho-recurrence} gives
\[
  u_{n+2}
  =u_{n+1}\frac{\alpha u_{n+1}+\beta u_n}
  {\gamma u_{n+1}+\delta u_n}
  =T_{u_{n+1}}(u_n).
\]
Thus every distinct adjacent pair counted by $L$ is an internal transition for the value set $V=\{u_0,\ldots,u_K\}$.  Apply \cref{thm:abstract-intro}.
\end{proof}

A broad source of ratio dynamics is a linear-fractional function of the index.

\begin{proposition}\label{prop:hypergeometric-ratio}
Let
\[
  f(t)=\frac{At+B}{Ct+D},
  \qquad AD-BC\ne0.
\]
Then
\[
  \varphi=f\circ(t\mapsto t+1)\circ f^{-1}
\]
is represented, up to a nonzero scalar, by
\begin{equation}\label{eq:phi-conjugate}
  \begin{pmatrix}
    AD-AC-BC&A^2\\
    -C^2&AC+AD-BC
  \end{pmatrix}.
\end{equation}
The determinant of this matrix is $(AD-BC)^2$.
Consequently, if a nonzero sequence satisfies
\begin{equation}\label{eq:hypergeometric-ratio}
  \frac{u_{n+1}}{u_n}=f(n),
\end{equation}
then its consecutive ratios obey a M\"obius recurrence.  Whenever all four entries in \eqref{eq:phi-conjugate} are nonzero modulo $p$, \cref{cor:ratio-dynamics} applies.
\end{proposition}

\begin{proof}
Represent $f$ and the translation $t\mapsto t+1$ by
\[
  F=\begin{pmatrix}A&B\\C&D\end{pmatrix},
  \qquad
  U=\begin{pmatrix}1&1\\0&1\end{pmatrix}.
\]
Then $\varphi$ is represented by $FUF^{-1}$.  Multiplication gives \eqref{eq:phi-conjugate}.  The final assertion follows from
\[
  \rho_{n+1}=f(n+1)=\varphi(f(n))=\varphi(\rho_n).
\]
\end{proof}

This proposition includes many fixed-parameter hypergeometric products.  It also clarifies why the transformation coefficients in the next two sections are simple: they are obtained by conjugating a translation.

\section{Fixed rows of Pascal's triangle}\label{sec:binomial}

For a prime $p$ and $0\le N<p$, define
\begin{equation}\label{eq:B-Np}
  B_{N,p}=\left\{\binom Nn\bmod p:0\le n\le N\right\}.
\end{equation}
When $N<p$, every element in the row is nonzero modulo $p$.  The digital distribution results in \cite{GarfieldWilf1992,BarbolosiGrabner1996,BaratGrabner2001} concern rows of arbitrary size and the multiplicities of residue classes throughout Pascal's triangle.  Here the restriction $N<p$ removes the zero entries and exposes a local transition structure inside one row.

Let
\[
  u_n=\binom Nn,
  \qquad 0\le n\le N.
\]
Then
\begin{equation}\label{eq:binomial-ratio}
  \rho_n=\frac{u_{n+1}}{u_n}=\frac{N-n}{n+1}.
\end{equation}
The next identity is the central algebraic observation for a fixed row.

\begin{lemma}\label{lem:binomial-ratio}
For $0\le n\le N-2$,
\begin{equation}\label{eq:binomial-mobius}
  \rho_{n+1}=\frac{N\rho_n-1}{\rho_n+N+2}.
\end{equation}
Equivalently,
\begin{equation}\label{eq:binomial-transition}
  u_{n+2}
  =u_{n+1}\frac{Nu_{n+1}-u_n}
  {u_{n+1}+(N+2)u_n}.
\end{equation}
\end{lemma}

\begin{proof}
Solving \eqref{eq:binomial-ratio} for $n$ gives
\[
  n=\frac{N-\rho_n}{\rho_n+1}.
\]
Substitution into
\[
  \rho_{n+1}=\frac{N-n-1}{n+2}
\]
gives \eqref{eq:binomial-mobius}.  Multiplying by $u_{n+1}$ gives \eqref{eq:binomial-transition}.
\end{proof}

For \eqref{eq:binomial-mobius}, the coefficients are
\begin{equation}\label{eq:binomial-coefficients}
  \alpha=N,
  \qquad \beta=-1,
  \qquad \gamma=1,
  \qquad \delta=N+2,
\end{equation}
and
\begin{equation}\label{eq:binomial-determinant}
  \alpha\delta-\beta\gamma=(N+1)^2.
\end{equation}
Thus the nondegeneracy assumptions hold precisely for
\[
  1\le N\le p-3.
\]

\begin{proof}[Proof of \cref{thm:binomial-intro}]
Assume first that $1\le N\le p-3$.  The ratios in \eqref{eq:binomial-ratio} are pairwise distinct.  Indeed, if $\rho_n=\rho_m$, then cross-multiplication gives
\[
  (N+1)(m-n)=0\quad\text{in }\F_p.
\]
Since $N+1\ne0$ and $0\le m,n<p$, one has $m=n$.

Consequently, the ordered pairs
\[
  (u_n,u_{n+1}),
  \qquad 0\le n\le N-2,
\]
are distinct.  There are $N-1$ of them, and \cref{cor:ratio-dynamics}, together with \eqref{eq:binomial-coefficients}--\eqref{eq:binomial-determinant}, gives
\[
  |B_{N,p}|\gg (N-1)^{1/2+\eta}.
\]
After adjusting the constant for $N\le2$, this is \eqref{eq:binomial-intro}.

For $N=p-2$, the standard negative-binomial identity gives
\[
  \binom{p-2}{n}\equiv\binom{-2}{n}=(-1)^n(n+1)\pmod p.
\]
As $n+1$ runs from $1$ to $p-1$, the values are exactly the odd representatives
\[
  1,3,5,\ldots,p-2,
\]
each occurring twice.  This proves \eqref{eq:pminus2-intro}.

Finally,
\[
  \binom{p-1}{n}\equiv(-1)^n\pmod p,
\]
which proves \eqref{eq:pminus1-intro}.
\end{proof}

\begin{remark}
The degeneration at $N=p-1$ is genuine.  In the M\"obius recurrence, the determinant is $(N+1)^2$, which vanishes exactly on that row; the value set then collapses to two elements.  The row $N=p-2$ has a different coefficient degeneration but remains large for the elementary reason displayed above.
\end{remark}

\section{Catalan and central binomial values}\label{sec:catalan}

Let
\[
  C_n=\frac1{n+1}\binom{2n}{n}
  \qquad\text{and}\qquad
  M_n=\binom{2n}{n}.
\]
The congruential distribution of these sequences has been studied through character sums, solution counts and automata; see \cite{GaraevLucaShparlinski2006,GaraevLucaShparlinski2007,Burns2017}.  The earlier covering result of Garaev--Luca--Shparlinski concerns an initial segment of length at most $p^{13/2}(\log p)^6$, whereas we consider the much shorter range
\[
  0\le n\le\frac{p-1}{2}.
\]
In this range $2n<p$, so both $C_n$ and $M_n$ are nonzero modulo $p$.

\subsection{Catalan numbers}

The adjacent ratio is
\begin{equation}\label{eq:catalan-ratio}
  \rho_n=\frac{C_{n+1}}{C_n}
  =\frac{2(2n+1)}{n+2}.
\end{equation}

\begin{lemma}\label{lem:catalan-ratio}
For every admissible $n$,
\begin{equation}\label{eq:catalan-mobius}
  \rho_{n+1}=\frac{2\rho_n+16}{-\rho_n+10}.
\end{equation}
Equivalently,
\begin{equation}\label{eq:catalan-transition}
  C_{n+2}
  =C_{n+1}\frac{2C_{n+1}+16C_n}{-C_{n+1}+10C_n}.
\end{equation}
\end{lemma}

\begin{proof}
This is \cref{prop:hypergeometric-ratio} with
\[
  f(t)=\frac{4t+2}{t+2}.
\]
Alternatively, solving \eqref{eq:catalan-ratio} for $n$ and substituting into the formula for $\rho_{n+1}$ gives \eqref{eq:catalan-mobius} directly.
\end{proof}

The coefficient quadruple is
\begin{equation}\label{eq:catalan-coefficients}
  (\alpha,\beta,\gamma,\delta)=(2,16,-1,10),
  \qquad
  \alpha\delta-\beta\gamma=36.
\end{equation}
It is nondegenerate for every prime $p\ge7$.

\subsection{Central binomial coefficients}

For $M_n=\binom{2n}{n}$,
\begin{equation}\label{eq:central-ratio}
  \tau_n=\frac{M_{n+1}}{M_n}
  =\frac{2(2n+1)}{n+1}.
\end{equation}

\begin{lemma}\label{lem:central-ratio}
For every admissible $n$,
\begin{equation}\label{eq:central-mobius}
  \tau_{n+1}=\frac{2\tau_n-16}{\tau_n-6}.
\end{equation}
Equivalently,
\begin{equation}\label{eq:central-transition}
  M_{n+2}
  =M_{n+1}\frac{2M_{n+1}-16M_n}{M_{n+1}-6M_n}.
\end{equation}
\end{lemma}

\begin{proof}
Apply \cref{prop:hypergeometric-ratio} to
\[
  f(t)=\frac{4t+2}{t+1}.
\]
\end{proof}

Here one may take
\begin{equation}\label{eq:central-coefficients}
  (\alpha,\beta,\gamma,\delta)=(2,-16,1,-6),
  \qquad
  \alpha\delta-\beta\gamma=4.
\end{equation}
The quadruple is nondegenerate for every prime $p\ge5$.

\begin{proof}[Proof of \cref{thm:catalan-intro}]
Put
\[
  K=\frac{p-1}{2}.
\]
For the Catalan sequence, the map
\[
  n\longmapsto\rho_n=\frac{4n+2}{n+2}
\]
is a M\"obius transformation of determinant $6$.  It is therefore injective on $\F_p$ when $p\ge7$.  Hence the adjacent pairs
\[
  (C_n,C_{n+1}),
  \qquad 0\le n\le K-2,
\]
are distinct.  Their number is
\[
  K-1=\frac{p-3}{2}.
\]
The coefficients in \eqref{eq:catalan-coefficients} satisfy the hypotheses of \cref{cor:ratio-dynamics}, which proves \eqref{eq:catalan-intro}.

For the central binomial sequence, the map
\[
  n\longmapsto\tau_n=\frac{4n+2}{n+1}
\]
has determinant $2$ and is injective for odd $p$.  The same number of adjacent pairs is therefore distinct.  Applying \cref{cor:ratio-dynamics} with \eqref{eq:central-coefficients} proves \eqref{eq:central-intro}.
\end{proof}

\section{Further remarks}

\subsection{General hypergeometric blocks}

Suppose that a nonzero sequence over $\F_p$ satisfies
\begin{equation}\label{eq:general-hypergeometric-final}
  \frac{u_{n+1}}{u_n}=\frac{An+B}{Cn+D}
\end{equation}
on an interval of indices, with $AD-BC\ne0$.  The matrix in \eqref{eq:phi-conjugate} gives the exact M\"obius recurrence for consecutive ratios.  If its four entries are nonzero and the ratios are distinct on the interval, \cref{cor:ratio-dynamics} yields a fixed power improvement over the square-root bound for the corresponding value set.  This includes many fixed-parameter Pochhammer quotients and hypergeometric products.

When one or more entries of the conjugated matrix vanish, the family may move into an affine subgroup or acquire a common projective configuration.  Such degenerations are not necessarily harder; rather, they require a different incidence input.  The affine ratio recurrence \eqref{eq:affine-ratio-recurrence} is the most useful example: it belongs to the reciprocal-affine branch and gives the explicit exponent $8/15$.

\subsection{What the two theorems abstract}

Both proofs have the same outer architecture:
\begin{enumerate}[label=(\roman*)]
  \item a structured sequence supplies many distinct internal transitions;
  \item Cauchy--Schwarz pairs transitions with a common input;
  \item eliminating that input produces a family of relative transformations;
  \item the relative family has low parameter multiplicity;
  \item an incidence theorem controls the resulting second moment.
\end{enumerate}
The difference lies in the geometry of the quotient family.  In the reciprocal-affine case it lives in $\Aff(1)$ and is handled by a Cartesian-product point-line theorem.  In the genuine M\"obius case it occupies a subgroup-escaping part of $\PSL_2(\F_p)$ and is handled by group expansion.

A low-order recurrence by itself is not enough.  For example, a second-order linear recurrence may generate only translations after elimination, and a single translation can have many parameter representations.  This is the reason that Fibonacci-type recurrences do not automatically benefit from the present method.  The decisive object is the quotient family, not the recurrence order in isolation.

The self-power sequence remains outside both theorems.  Its physical coordinate and discrete-logarithmic coordinate are strongly coupled, and no fixed-degree reciprocal-affine or M\"obius transition family is currently known.  Thus the failure of the present method for $x^x$ occurs before the incidence estimate: the required low-dimensional quotient model has not yet been found.

\subsection{Possible quantitative refinements}

The affine exponent $8/15$ is explicit because it comes directly from the exponent $11/4$ in \eqref{eq:cartesian-special}.  Any improvement for the special line families arising from the sequences above would immediately improve the value-set exponent.

The M\"obius exponent $1/2+\eta$ is qualitative.  Quantitative versions of M\"obius incidence estimates, such as those in \cite{Shkredov2021,RudnevWheeler2022,WarrenWheeler2023}, may produce explicit exponents for particular coefficient ranges.  The exact injectivity of the quotient map and the linear subgroup-coset bound proved here provide favorable inputs for such a refinement.

\section*{Statement on the use of AI}

ChatGPT 5.6 Pro was used during exploratory work and in preparing an initial draft, including algebraic calculations, literature searches, exposition, and LaTeX preparation.  The author is responsible for checking every argument and for the mathematical content of any submitted version.

\end{document}